\documentclass{article}

\usepackage{arxiv}
\usepackage[dvipsnames]{xcolor}
\usepackage[utf8]{inputenc} % allow utf-8 input
\usepackage[T1]{fontenc}    % use 8-bit T1 fonts
\usepackage{hyperref}       % hyperlinks
\usepackage{url}            % simple URL typesetting
\usepackage{booktabs}       % professional-quality tables
\usepackage{amsfonts}       % blackboard math symbols
\usepackage{nicefrac}       % compact symbols for 1/2, etc.
\usepackage{microtype}      % microtypography
\usepackage{lipsum}		% Can be removed after putting your text content
\usepackage{graphicx}
\usepackage[square,sort,comma,numbers]{natbib}
\usepackage{doi}
\usepackage{setspace}
\newcommand{\modspacing}{%
  \setstretch{1.18}%  default
}
\usepackage[font=onehalfspacing]{caption}
\usepackage{mathtools}
\usepackage[title]{appendix}
\usepackage{accents}
\usepackage{booktabs}
\usepackage{varwidth}

\newcommand*\underdot[1]{%
  \underaccent{\dot}{#1}}

\usepackage{graphicx}
\usepackage{mathptmx}
\usepackage{amsmath}

\usepackage{amssymb}

\usepackage{amsthm}
\usepackage{algorithm,algpseudocode}
\usepackage{microtype}
\DeclareMathAlphabet{\mathcal}{OMS}{cmsy}{m}{n} % to avoid ugly mathcals

\newtheorem{thm}{Theorem}[section]
\newtheorem{lem}[thm]{Lemma}

\newtheorem{prop}[thm]{Proposition}

\theoremstyle{definition}
\newtheorem{defn}{Definition}[section]

\title{Arbitrary-length analogs to de Bruijn sequences}

\author{\href{https://orcid.org/0000-0001-8145-1484}{\includegraphics[scale=0.06]{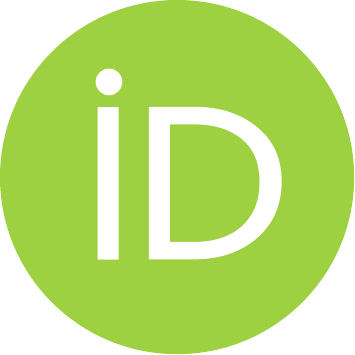}\hspace{1mm}Abhinav Nellore} \\
	Oregon Health \& Science University \\
	Portland, Oregon 97239, USA \\
	\texttt{anellore@gmail.com} \\
	\And
	\href{https://orcid.org/0000-0001-7651-089X}{\includegraphics[scale=0.06]{orcid.pdf}\hspace{1mm}Rachel Ward} \\
	The University of Texas at Austin\\
	Austin, Texas 78712, USA \\
	\texttt{rward@math.utexas.edu} \\
}

\renewcommand{\headeright}{}
\renewcommand{\undertitle}{A preprint}
\renewcommand{\shorttitle}{Arbitrary-length analogs to de Bruijn sequences}

\hypersetup{
pdftitle={Arbitrary-length analogs to de Bruijn sequences},
pdfsubject={math.CO},
pdfauthor={Abhinav Nellore, Rachel Ward}
}

\begin{document}
\modspacing
\maketitle

\begin{abstract}
Let $\widetilde{\alpha}$ be a length-$L$ cyclic sequence of characters from a size-$K$ alphabet $\mathcal{A}$ such that the number of occurrences of any length-$m$ string on $\mathcal{A}$ as a substring of $\widetilde{\alpha}$ is $\lfloor L / K^m \rfloor$ or $\lceil L / K^m \rceil$. When $L = K^N$ for any positive integer $N$, $\widetilde{\alpha}$ is a de Bruijn sequence of order $N$, and when $L \neq K^N$, $\widetilde{\alpha}$ shares many properties with de Bruijn sequences. We describe an algorithm that outputs some $\widetilde{\alpha}$ for any combination of $K \geq 2$ and $L \geq 1$ in $O(L)$ time using $O(L \log K)$ space. This algorithm extends Lempel's recursive construction of a binary de Bruijn sequence. An implementation written in Python is available at \url{https://github.com/nelloreward/pkl}.
\end{abstract}

\section{Introduction}
\label{sec:intro}

\subsection{Preliminaries}

This paper is concerned with necklaces, otherwise known as circular strings or circular words. A \textit{necklace} is a cyclic sequence of characters; each character has a direct predecessor and a direct successor, but no character begins or ends the sequence. So if $101$ is said to be a necklace, $011$ and $110$ refer to the same necklace. In the remainder of this paper, the term \textit{string} exclusively refers to a sequence of characters with a first character and a last character. A \textit{substring} of a necklace is a string of contiguous characters whose length does not exceed the necklace's length. So the set of length-$2$ substrings of the necklace $101$ is $\{10, 01, 11\}$. A \textit{rotation} of a necklace is a substring whose length is precisely the necklace's length, and a \textit{prefix} of a string is any substring starting at the string's first character. So $101$ can be called a rotation of a necklace, and $10$ is a prefix of that rotation.

A \textit{de Bruijn sequence} of order $N$ on a size-$K$ alphabet $\mathcal{A}$ is a length-$K^N$ necklace that includes every possible length-$N$ string on $\mathcal{A}$ as a substring \cite{saint1894solution, de1946combinatorial, de1951circuits, de1975acknowledgement}. There are $(K!)^{K^{N-1}}/K^N$ distinct de Bruijn sequences of order $N$ on $\mathcal{A}$ \cite{de1951circuits}. (See the \hyperref[sec:appendix]{appendix} for a brief summary of the curious history of de Bruijn sequences.) An example for $\mathcal{A} = \{0, 1\}$ and $N = 4$ is the length-$16$ necklace
\begin{equation*}
    0000110101111001\,.
\end{equation*}
A de Bruijn sequence of order $N$ on $\mathcal{A}$ is optimally short in the sense that its length is $K^N$, and there are $K^N$ possible length-$N$ strings on $\mathcal{A}$. But more is true: because any length-$m$ string on $\mathcal{A}$ is a prefix of each of $K^{N-m}$ strings on $\mathcal{A}$ when $m \leq N$, the sequence has precisely $K^{N-m}$ occurrences of that length-$m$ string as a substring. So in the example above, there are $8$ occurrences of $0$, $4$ occurrences of $00$, $2$ occurrences of $000$, and $1$ occurrence of $0000$. Note by symmetry, $K^{N-m}$ is also the expected number of occurrences of any length-$m$ string on $\mathcal{A}$ as a substring of a necklace of length $K^N$ formed by drawing each of its characters uniformly at random from $\mathcal{A}$. More generally, by symmetry, $L / K^m$ is the expected number of occurrences of any length-$m$ string on $\mathcal{A}$ for $m \leq L$ as a substring of a necklace of arbitrary length $L$ formed by drawing each of its characters uniformly at random from $\mathcal{A}$.

\subsection{\texorpdfstring{$P^{(K)}_L$}{P(K)L}-sequences}

Consider a necklace defined as follows.
\begin{defn}[$P^{(K)}_L$-sequence]\label{def:p}
A $P^{(K)}_L$-sequence is a length-$L$ necklace on a size-$K$ alphabet $\mathcal{A}$ such that the number of occurrences of any length-$m$ string on $\mathcal{A}$ for $m \leq L$ as a substring of the necklace is $\lfloor L/K^{m} \rfloor$ or $\lceil L/K^{m} \rceil$.
\end{defn}
This paper proves by construction that a $P^{(K)}_L$-sequence exists for any combination of $K \geq 2$ and $L \geq 1$, giving an algorithm for sequence generation that runs in $O(L)$ time using $O(L \log k)$ space.

When $L = K^N$ for any positive integer $N$, $\lfloor L/K^{m} \rfloor = \lceil L/K^{m} \rceil = K^{N-m}$ for $m \leq N$, and a $P^{(K)}_L$-sequence collapses to a de Bruijn sequence of order $N$. When $L \neq K^N$, a $P^{(K)}_L$-sequence is a natural interpolative generalization of a de Bruijn sequence: it is a necklace for which the number of occurrences of any length-$m$ string on $\mathcal{A}$ for $m \leq L$ as a substring differs by less than one from its expected value for a length-$L$ necklace formed by drawing each of its characters uniformly at random from $\mathcal{A}$. When this expected value is an integer, $\lfloor L/K^{m} \rfloor = \lceil L/K^{m} \rceil$, and the number of occurrences of any length-$m$ string on $\mathcal{A}$ as a substring of a given $P^{(K)}_L$-sequence is equal to the number of occurrences of any \textit{other} length-$m$ string on $\mathcal{A}$ as a substring of that sequence. When this expected value is not an integer, a $P^{(K)}_L$-sequence comes as close as it can to achieving the same end, as formalized in the proposition below.
\begin{prop}
Consider a $P^{(K)}_L$-sequence $\widetilde{\alpha}$. Load across length-$m$ strings on $\mathcal{A}$ for $m \leq L$ is balanced in $\widetilde{\alpha}$ as follows.
\begin{enumerate}
    \item When $L / K^m$ is an integer, each length-$m$ string on $\mathcal{A}$ occurs precisely $L / K^m$ times as a substring of $\widetilde{\alpha}$.
    \item When $L / K^m$ is not an integer, each of $L - K^m \lfloor L / K^m \rfloor$ length-$m$ strings on $\mathcal{A}$ occurs precisely $\lceil L/K^{m} \rceil$ times as a substring of $\widetilde{\alpha}$, and each of $K^m \lceil L / K^m \rceil - L$ length-$m$ strings on $\mathcal{A}$ occurs precisely $\lfloor L/K^{m} \rfloor$ times as a substring of $\widetilde{\alpha}$.
\end{enumerate}
\end{prop}
\begin{proof}
Item 1 is manifestly true from $\lfloor L/K^{m} \rfloor = \lceil L/K^{m} \rceil$. To see why item 2 is true, consider the system of Diophantine equations
\begin{align}
\begin{aligned}
    a \lfloor L / K^m \rfloor +  b \lceil L / K^m \rceil &= L \\
    a + b &= K^m
\end{aligned}\quad.
\end{align}
Above, $a$ represents the number of length-$m$ strings on $\mathcal{A}$ for which there are $\lfloor L / K^m \rfloor$ occurrences each as a substring of $\widetilde{\alpha}$, and $b$ represents the number of length-$m$ strings on $\mathcal{A}$ for which there are $\lceil L / K^m \rceil$ occurrences each as a substring of $\widetilde{\alpha}$. The first equation says the total number of occurrences of strings as substrings of $\widetilde{\alpha}$ is $L$, and the second says there is a total of $K^m$ length-$m$ strings on $\mathcal{A}$. Note the equations hold only when $L / K^m$ is nonintegral---that is, $\lfloor L / K^m \rfloor + 1 = \lceil L / K^m \rceil$. In this case, it is easily verified the unique solution to the system is $a = K^m \lceil L / K^m \rceil - L$ and $b = L - K^m \lfloor L / K^m \rfloor$.
\end{proof}
An example for $\mathcal{A} = \{0, 1\}$ and $L = 12$ is the sequence
\begin{equation}\label{eq:pex}
    000110111001 \,.
\end{equation}
To see why, note $L / K^m$ for $L = 12$ and $K = 2$ is $6$ for $m=1$, $3$ for $m=2$, between $1$ and $2$ for $m=3$, and between $0$ and $1$ for any $m \geq 4$. Further, the sequence \eqref{eq:pex} contains, as a substring, precisely
\begin{enumerate}
    \item $6$ occurrences of each string in the set $\{0, 1\}$;
    \item $3$ occurrences of each string in the set $\{00, 01, 10, 11\}$;
    \item $2$ occurrences of each string in the set $\{001, 011, 100, 110\}$, which is of size $L - K^m \lfloor L / K^m \rfloor = 12 - 2^3 \lfloor 12 / 2^3 \rfloor = 4$, and $1$ occurrence of each string in the set $\{000, 010, 101, 111\}$, which is of size $K^m \lceil L / K^m \rceil - L = 2^3 \lceil 12 / 2^3 \rceil - 12 = 4$;
    \item $1$ occurrence of each string in the set $$M := \{0001, 0011, 0110, 1101, 1011, 0111, 1110, 1100, 1001, 0010, 0100, 1000\}\,,$$ which is of size $L - K^m \lfloor L / K^m \rfloor = 12 - 2^4 \lfloor 12 / 2^4 \rfloor = 12$, and $0$ occurrences of each of the set of length-$4$ strings on $\mathcal{A}$ not in $M$, which is of size $K^m \lceil L / K^m \rceil - L = 2^4 \lceil 12 / 2^4 \rceil - 12 = 4$; and
    \item $0$ or $1$ occurrences of any length-$m$ string for $4 < m \leq L$ due to item 4 above.
\end{enumerate}

\subsection{\texorpdfstring{$P^{(K)}_L$}{P(K)L}-sequences vs. other de Bruijn-like sequences}

Two other arbitrary-length generalizations of de Bruijn sequences have appeared in the literature:
\begin{enumerate}
\item What we call a \textit{Lempel-Radchenko sequence} is a length-$L$ necklace on a size-$K$ alphabet $\mathcal{A}$ such that every length-$\lceil \log_K L \rceil$ string on $\mathcal{A}$ has at most one occurrence as a substring of the necklace. As recounted by Yoeli in \cite{yoeli1963counting}, according to Radchenko and Filippov in \cite{radchenko1959shifting}, the existence of binary Lempel-Radchenko sequences of any length was first proved by Radchenko in his unpublished 1958 University of Leningrad PhD dissertation \cite{radchenko1958code}. Other binary-case existence proofs were furnished by (1) Yoeli himself in \cite{yoeli1961nonlinear} and \cite{yoeli1962binary}; (2) Bryant, Heath, and Killik in \cite{bryant1962counting} based on the work \cite{heath1961chain} of Heath and Gribble; and (3) Golomb, Welch, and Goldstein in \cite{golomb1959cycles}. Explicit constructions of arbitrary-length binary Lempel-Radchenko sequences were given by Etzion in 1986 \cite{etzion1986algorithma}. In brief, Etzion's approach is to join necklaces derived from the pure cycling register, potentially overshooting the target length $L$, and subsequently remove substrings as necessary in the resulting sequence according to specific rules to achieve the target length. This takes $o(\log L)$ time per bit generated and uses $O(L \log L)$ space.

The existence of Lempel-Radchenko sequences of any length for any alphabet size was proved in 1971 by Lempel in \cite{lempel1971m}. In the special case where the alphabet size is a power of a prime number, one of two approaches for sequence construction effective at any length $L$ may be used: either (1) pursue the algebraic construction described by Hemmati and Costello in their 1978 paper \cite{hemmati1978algebraic}, or (2) cut out a length-$L$ stretch of contiguous sequence from a de Bruijn sequence longer than $L$ generated by a linear feedback function, as described in Chapter 7, Section 5 of Golomb's text \cite{golomb2017shift}. In his 2000 paper \cite{landsberg2000feedback}, Landsberg built on Golomb's technique, explaining in the appendix how to use it to construct a Lempel-Radchenko sequence on an alphabet of arbitrary size. The idea is to decompose the desired alphabet size into a product of powers of pairwise-distinct primes, construct length-$L$ sequences on alphabets of sizes equal to factors in this product with Golomb's technique, and finally write a particular linear superposition of the sequences. The time and space requirements of Hemmati and Costello's construction, when optimized, have gone unstudied in the literature. In general, Golomb's technique gives a length-$L$ Lempel-Radchenko sequence in $O(L \log L)$ time using $O(\log L)$ space, and Landsberg's generalization multiplies these complexities by the number of factors in the prime power decomposition of the alphabet size. Etzion suggested in his 1986 paper \cite{etzion1986algorithma} that, using results from \cite{etzion1986algorithmb}, his algorithm generating a binary Lempel-Radchenko sequence could be extended to generate a Lempel-Radchenko sequence for any alphabet size, but he did not do so explicitly. It is reported on Joe Sawada's website \cite{sawadasite} that in their recent unpublished manuscript \cite{gundogan2021cut}, G\"undo\v{g}an, Sawada, and Cameron extend Etzion's construction to arbitrary alphabet sizes, streamlining it so it generates each character in $O(\log L)$ time using $O(\log L)$ space. Sawada's website further includes an implementation in C.

\item A \textit{generalized de Bruijn sequence} is a length-$L$ Lempel-Radchenko sequence on a size-$K$ alphabet $\mathcal{A}$ such that every length-$\lfloor \log_K L \rfloor$ string on $\mathcal{A}$ is a substring of the sequence. Generalized de Bruijn sequences were recently introduced by Gabric, Hollub, and Shallit in \cite{gabric2019generalized, gabric2021maximal}. These papers also prove generalized de Bruijn sequences exist for any combination of $L \geq 1$ and $K \geq 2$. No work to date has given explicit constructions of arbitrary-length generalized de Bruijn sequences.
\end{enumerate}

We prove the following.
\begin{thm}
A $P^{(K)}_L$-sequence is a generalized de Bruijn sequence and therefore also a Lempel-Radchenko sequence.
\end{thm}
\begin{proof}
Let $\widetilde{\alpha}$ be a $P^{(K)}_L$-sequence. The proposition is true if and only if
\begin{enumerate}
    \item every length-$\lceil \log_K L \rceil$ substring of $\widetilde{\alpha}$ occurs precisely once in the sequence, and
    \item every length-$\lfloor \log_K L \rfloor$ string on $\mathcal{A}$ is a substring of $\widetilde{\alpha}$.
\end{enumerate}
Item 1 is true because from Definition \ref{def:p}, $\widetilde{\alpha}$ has
\begin{equation*}
\lfloor L / K^{\lceil \log_K L \rceil} \rfloor =
\begin{cases}
1 & \mbox{when } \log_K L \mbox{ is an integer} \\
0 & \mbox{otherwise} \\
\end{cases}
\end{equation*}
or $\lceil L / K^{\lceil \log_K L \rceil} \rceil = 1$ occurrences of any length-$\lceil \log_K L \rceil$ string on $\mathcal{A}$ as a substring. Item 2 is true because from Definition \ref{def:p}, $\widetilde{\alpha}$ has
\begin{equation*}
\lceil L / K^{\lfloor \log_K L \rfloor} \rceil =
\begin{cases}
1 & \mbox{when } \log_K L \mbox{ is an integer} \\
k & \mbox{otherwise for }k \in \{2, \ldots, K\} \\
\end{cases}
\end{equation*} 
or $\lfloor L / K^{\lfloor \log_K L \rfloor} \rfloor \in \{1, \ldots, K-1\}$ occurrences of any length-$\lfloor \log_K L \rfloor$ string on $\mathcal{A}$ as a substring.
\end{proof}
$P^{(K)}_L$-sequences are more tightly constrained than generalized de Bruijn sequences and Lempel-Radchenko sequences. A length-$L$ Lempel-Radchenko sequence imposes no requirements regarding presence or absence of particular strings as substrings; it simply requires that the number of distinct length-$\lceil \log_K L \rceil$ substrings is $L$. A length-$L$ generalized de Bruijn sequence on $\mathcal{A}$ goes a step further, requiring not only this distinctness, but also the presence of every string on $\mathcal{A}$ smaller than $\lceil \log_K L \rceil$ as a substring. A $P^{(K)}_L$-sequence goes yet another step further, requiring not only this presence, but also specific incidences of strings as substrings that, as best as they can, try not to bias the sequence toward inclusion of any one length-$m$ string over another. This requirement makes $P^{(K)}_L$-sequences, in general, more de Bruijn-like than Lempel-Radchenko sequences and generalized de Bruijn sequences.

An example (borrowed from \cite{gabric2021maximal}) of a Lempel-Radchenko sequence that is not a generalized de Bruijn sequence and therefore also not a $P^{(K)}_L$-sequence for $\mathcal{A} = \{0, 1\}$ and $L = 11$ is
\begin{equation}\label{eq:shiftex}
10011110000\,.
\end{equation}
In this case, $\lceil \log_K L \rceil = \lceil \log_2 11 \rceil = 4$, and indeed, there is precisely one occurrence in \eqref{eq:shiftex} of every length-$4$ substring of \eqref{eq:shiftex}. But $\lfloor \log_K L \rfloor = \lfloor \log_2 11 \rfloor = 3$, and in \eqref{eq:shiftex} just $7$ of $8$ length-$3$ strings on $\mathcal{A}$ occur as substrings; the sequence is missing $101$. An example of a generalized de Bruijn sequence that is not a $P^{(K)}_L$-sequence for $\mathcal{A} = \{0, 1, 2\}$ and $L = 12$ is
\begin{equation}\label{eq:genex}
    000111101011\,.
\end{equation}
Again, $\lceil \log_K L \rceil = \lceil \log_2 12 \rceil = 4$ and $\lfloor \log_K L \rfloor = \lfloor \log_2 12 \rfloor = 3$. Now, not only is there precisely one occurrence in \eqref{eq:genex} of every length-$4$ substring of \eqref{eq:genex}, but also all $8$ length-$3$ strings on $\mathcal{A}$ occur as substrings. However, \eqref{eq:genex} should have $\lceil L / K \rceil = \lfloor L / K \rfloor = 12 / 2 = 6$ occurrences of each of $1$ and $0$ as substrings to be a $P^{(K)}_L$-sequence, and it has $5$ occurrences of $0$ and $7$ occurrences of $1$. This imbalance of $0$s and $1$s leads to further violations of constraints on $P^{(K)}_L$-sequences at other substring lengths. Another example of a generalized de Bruijn sequence that is not a $P^{(K)}_L$-sequence, this time on the nonbinary alphabet $\mathcal{A} = \{0, 1, 2\}$ and for $L = 20$, is
\begin{equation}\label{eq:landsex}
    02220010121120111002\,.
\end{equation}
(This sequence was constructed by Landsberg in \cite{landsberg2000feedback} using Golomb's technique from \cite{golomb2017shift} as an example of a Lempel-Radchenko sequence.) Note $\lceil \log_K L \rceil = \lceil \log_3 20 \rceil = 3$ and $\lfloor \log_K L \rfloor = \lfloor \log_3 20 \rfloor = 2$, every length-$3$ substring occurs exactly once, and every length-$2$ string on $\mathcal{A}$ is present as a substring. But \eqref{eq:landsex} should have, for $m = 2$, precisely $\lceil L / K^m \rceil = 20 / 3^2 = 3$ or $\lfloor L / K^m \rfloor = 20 / 3^2 = 2$ occurrences of every length-$2$ string on $\mathcal{A}$ as a substring, and there is only $1$ occurrence of $21$ as a substring of \eqref{eq:landsex}.

\subsection{de Bruijn sequence constructions vs. de Bruijn-like sequence constructions}

Unlike the current situation with de Bruijn-like sequences of arbitrary length, there is a veritable cornucopia of elegant constructions of de Bruijn sequences. Excellent summaries of many of these are provided on Sawada's website \cite{sawadasite}. They include
\begin{enumerate}
    \item greedy constructions. Prominent examples are the prefer-largest/prefer-smallest \cite{martin1934problem}, prefer-same \cite{eldert1958shifting, fredricksen1982survey, alhakim2021revisiting}, and prefer-opposite \cite{alhakim2010simple} algorithms;
    \item shift rules. A shift rule maps a length-$N$ substring of a de Bruijn sequence of order $N$ to the next length-$N$ substring of the sequence. Shift rules are often simple, economical, and efficient; examples generating each character of a de Bruijn sequence in amortized constant time using $O(N)$ space are \cite{sawada2016surprisingly, huang1990new, etzion1984algorithms, fredricksen1975class} in the binary case and \cite{sawada2017simple} in the $K$-ary case. See \cite{amram2019efficient, jansen1991efficient, gabric2018framework, gabric2019successor, yang1992construction, chang2019general, zhu2021efficiently} for other specific rules;
    \item concatenation rules. The best-known example, obtained by Fredricksen and Maiorana in 1978 \cite{fredricksen1978necklaces}, joins all Lyndon words on an ordered alphabet of size $K$ whose lengths divide the desired order $N$ in lexicographic order to form the lexicographically smallest (i.e., ``granddaddy'') de Bruijn sequence of that order on that alphabet. (Also see \cite{ford1957cyclic} for Ford's independent work generating this sequence.) The sequence is obtained in amortized constant time per character using $O(N)$ space with the efficient Lyndon word generation approach of Ruskey, Savage, and Wang \cite{ruskey1992generating}, which builds on Fredricksen, Kessler, and Maiorana's papers \cite{fredricksen1978necklaces, fredricksen1986algorithm}. Dragon, Hernandez, Sawada, Williams, and Wong recently discovered that joining the Lyndon words in colexicographic order instead also outputs a particular de Bruijn sequence, the ``grandmama'' sequence \cite{dragon2016grandmama, dragon2018constructing}. A generic concatenation approach using colexicographic order is developed in \cite{gabric2017bruijn, gabric2018constructing};
    \item recursive constructions. Broadly, these approaches are based on transforming a de Bruijn sequence into a de Bruijn sequence of higher order, where the transformation can be implemented recursively. They fall into two principal classes:
    \begin{enumerate}
    \item the constructions of Mitchell, Etzion, and Paterson in \cite{mitchell1996method}, which interleave punctured and padded variants of a binary de Bruijn sequence of order $N$ and modify the result slightly to obtain a binary de Bruijn sequence of order $2N$. If starting with a known binary de Bruijn sequence, this process takes amortized $O(1)$ time per output bit while using $O(1)$ additional space. The constructions are notable for being efficiently decodable---that is, the position of any given string on $\mathcal{A}$ occurring exactly once in the sequence as a substring can be retrieved in time polynomial in $N$;
    \item constructions based on Lempel's D-morphism (otherwise known as Lempel's homomorphism) \cite{lempel1970homomorphism}, whose inverse lifts any length-$L$ necklace $\widetilde{\beta}$ on a size-$K$ alphabet $\mathcal{A}$ to up to $K$ necklaces on $\mathcal{A}$. When $\widetilde{\beta}$ is a de Bruijn sequence of order $N$, the necklaces to which it is lifted may be joined to form a de Bruijn sequence of order $N+1$. Efficient implementations constructing binary de Bruijn sequences of arbitrary order by repeated application of Lempel's D-morphism are given by Annexstein \cite{annexstein1997generating} as well as Chang, Park, Kim, and Song \cite{chang1999efficient}; in general, a length-$L$ binary de Bruijn sequence is generated in $O(L)$ time using $O(L)$ space. Lempel confined attention to the binary case in \cite{lempel1970homomorphism}. An extension to alphabets of arbitrary size was first written by Ronse in \cite{ronse1984feedback} and also developed by Tuliani in \cite{tuliani2001bruijn}; it was further generalized by Alhakim and Akinwande in \cite{alhakim2011recursive}. See \cite{games1983generalized, alhakim2017stretching} for other generalizations as well as \cite{tuliani2001bruijn} for a decodable de Bruijn sequence construction exploiting both interleaving and Lempel's D-morphism.
    \end{enumerate}
\end{enumerate}

It is possible construction techniques for de Bruijn sequences have been more easily uncovered than for their arbitrary-length cousins as traditionally defined precisely because de Bruijn sequences are more tightly constrained. But $P^{(K)}_L$-sequences are similarly constrained.

\subsection{Our contribution}

This paper introduces the concept of $P^{(K)}_L$-sequences. Further, it extends recursive de Bruijn sequence constructions based on Lempel's D-morphism \cite{lempel1970homomorphism, ronse1984feedback, tuliani2001bruijn, alhakim2011recursive}, giving an algorithm that outputs a $P^{(K)}_L$-sequence on the alphabet $\{0, \ldots, K - 1\}$ for any combination of $L \geq 1$ and $K \geq 2$ in $O(L)$ time using $O(L \log K)$ space. The essence of our approach is to lengthen each of $d_i$ longest runs of the same nonzero character by a single character at the $i$th step before lifting, where the $\{ d_i \}$ are the digits of the desired length $L$ of the $P^{(K)}_L$-sequence when expressed in base $K$---that is, for $L = \sum_{i=0}^{\lfloor \log_K L \rfloor} d_i K^{\lfloor \log_K L \rfloor - i}$. Finally, this paper is accompanied by Python code at \url{https://github.com/nelloreward/pkl} implementing our algorithm.

We were motivated to study arbitrary-length generalizations of de Bruijn sequences by \cite{nellore2021invertible}, which introduces nength, an analog to the Burrows-Wheeler transform \cite{burrows1994block} for offline string matching in labeled digraphs. In a step preceding the transform, a digraph with edges labeled on one alphabet is augmented with a directed cycle that (1) includes every vertex of the graph and (2) matches a de Bruijn-like sequence on a different alphabet. This vests each vertex with a unique tag along the cycle. But if the de Bruijn-like sequence is an arbitrary Lempel-Radchenko or generalized de Bruijn sequence, some vertices may be significantly more identifiable than others when locating matches to query strings in the graph using its nength, biasing performance. So in general, it is reasonable to arrange that the directed cycle matches a $P^{(K)}_L$-sequence, which distributes identifiability across vertices as evenly as possible.

The remainder of this paper is organized as follows. The next section develops our algorithm for generating \mbox{$P^{(K)}_L$-sequences}, proving space and performance guarantees. The third and final section lists some open questions.

\section{Generating \texorpdfstring{$P^{(K)}_L$}{P(K)L}-sequences}

\subsection{Additional notation and conventions}\label{sec:notation}

In the development that follows, necklaces are represented by lowercase Greek letters adorned with tildes such as $\widetilde{\beta}$ and $\widetilde{\gamma}$, and strings are represented by unadorned lowercase Greek letters such as $\omega$ and $\xi$. A necklace or string's length or a set's size is denoted using $| \cdot |$. So $|\widetilde{\beta}|$ is the length of the necklace $\widetilde{\beta}$, and $|V|$ is the size of the set $V$. Necklaces and strings may be in indexed families, where for example in $\widetilde{\beta}_i$, $i$ specifies the family member. Further, a necklace or string may be written as a function of another necklace or string. So $\omega(\widetilde{\beta})$ denotes that the string $\omega$ is a function of the necklace $\widetilde{\beta}$. When any function's argument is clear from context, that argument may be dropped with prior warning. So $\omega(\widetilde{\beta})$ may be written as, simply, $\omega$.

The operation of \textit{joining} two necklaces $\widetilde{\beta}$ and $\widetilde{\gamma}$ at a string $\omega$ to form a new necklace $\widetilde{\lambda}$ refers to cycle joining, described in Chapter 6 of Golomb's text \cite{golomb2017shift}. $\widetilde{\lambda}$ is obtained by concatenating rotations of $\widetilde{\beta}$ and $\widetilde{\gamma}$ that share the prefix $\omega$. So if $\widetilde{\beta} = 00101101$ and $\widetilde{\gamma} = 0110001$, joining $\widetilde{\beta}$ and $\widetilde{\gamma}$ at $110$ gives $\widetilde{\lambda} = 110100101100010$. There may be more than one occurrence of $\omega$ as a substring of at least one of $\widetilde{\beta}$ and $\widetilde{\gamma}$, so there may be more than one way to join them at $\omega$. Any way is permitted in such a case. Note joining $\widetilde{\beta}$ and $\widetilde{\gamma}$ at $\omega$ preserves length-$m$ substring occurrence frequencies for $m \leq |\omega|+1$.

For any positive integer $j$,
$$
[j] := \{0, 1, \ldots, j - 1\}\,.
$$
While results are obtained for sequences on the alphabet $[K]$ here, they may be translated to any size-$K$ alphabet $\mathcal{A}$ by appropriate substitution of characters. When a string or necklace is initially declared to be on the alphabet $[K]$, but an expression $y$ for one of its characters is written such that $y \notin [K]$, that character should be interpreted as $y - K \lfloor y / K\rfloor$. This is simply the remainder of floored division of $y$ by $K$. Put another way, expressions for characters of strings on $[K]$ respect arithmetic modulo $K$. For example, if the first character of a string on the alphabet $[2] = \{0, 1\}$ is specified as an expression that equals $9$, that character is $1$.

Individual characters comprising strings are often expressed in terms of variables, so a necklace or string may be written as a comma-separated list of characters enclosed by parentheses, where in the necklace case, $\circlearrowright$ is included as a subscript. For example, for $i = 3$, if $(i, i + 1, 0, 1)$ is said to be on the alphabet $[4]$, it is the string $3001$, while if $(i, i + 1, 0, 1)_\circlearrowright$ is said to be on $[4]$, it is the necklace $3001$. Bracket notation is used to refer to a specific character of a string or necklace. So $\omega[i]$ refers to the character at index $i$ of $\omega$. Further, characters of a string are indexed in order, so $\omega[i+1]$ appears directly after $\omega[i]$ in $\omega$. $\omega[0]$ and $\omega[|\omega| - 1]$ refer, respectively, to the first and last characters of the string $\omega$. For a necklace, the choice of the character at index $0$ is arbitrary, but in a parenthetical representation of that necklace, the character at index $0$ always comes first. So an arbitrary length-$L$ necklace $\widetilde{\beta}$ always equals
$$
\widetilde{\beta} = (\widetilde{\beta}[0], \widetilde{\beta}[1], \ldots, \widetilde{\beta}[L - 1])_\circlearrowright
$$
but not necessarily
$$
\widetilde{\beta} = (\widetilde{\beta}[1], \widetilde{\beta}[2], \ldots, \widetilde{\beta}[L - 1], \widetilde{\beta}[0])_\circlearrowright\,.
$$
A valid character index of a string $\omega$ is confined to $[|\omega|]$, but a valid character index of a length-$L$ necklace $\widetilde{\beta}$ is any integer $j$, with the stipulation
$$
\widetilde{\beta}[j] = \widetilde{\beta}[j + L]\,.
$$
A string or necklace on $[K]$ can be summed with any integer by adding that integer to each of its characters modulo $K$. So for an integer $j$ and a length-$L$ necklace $\widetilde{\beta}$,
$$
    \widetilde{\beta} + j = j + (\widetilde{\beta}[0], \widetilde{\beta}[1], \ldots, \widetilde{\beta}[L - 1])_\circlearrowright = (\widetilde{\beta}[0] + j, \widetilde{\beta}[1] + j, \ldots, \widetilde{\beta}[L - 1] + j)_\circlearrowright\,.
$$

Finally, $\mathbf{i}_{m}$ is used as a shorthand for the length-$m$ string $(i, i, \ldots, i)$, $\mathbf{i}_{m}^{++}$ is used as a shorthand for the length-$m$ string $(i, i+1, \ldots, i+m-1)$, and  $\widetilde{\mathbf{i}}_{m}^{++}$ is used as a shorthand for the length-$m$ necklace $(i, i+1, \ldots, i+m-1)_\circlearrowright$. In a slight abuse of notation, a variable representing a string such as $\omega$, $\mathbf{i}_m$, or $\mathbf{i}_{m}^{++}$ can take the place of a character in a parenthetical representation of a string or necklace. So if $(\mathbf{2}_{6}^{++}, 3)$ is said to be a substring of a string on the alphabet $[4]$, that substring is $2301233$.

\subsection{Lempel's lift}\label{sec:lempelslift}
Lempel's lift, defined below, realizes the simplest $K$-ary version of Lempel's D-morphism \cite{lempel1970homomorphism, ronse1984feedback, tuliani2001bruijn, alhakim2011recursive} in inverse form.
\begin{defn}[Lempel's lift]
\label{def:lempelslift}
Consider a length-$L$ necklace $\widetilde{\beta}$ on the alphabet $[K]$. Lempel's lift of $\widetilde{\beta}$, denoted by $\{\widetilde{\lambda}_i(\widetilde{\beta})\}$, is the indexed family of necklaces on $[K]$ specified by
\begin{equation}\label{eq:lambdadef}
\widetilde{\lambda}_i(\widetilde{\beta}) = i + \left(\widetilde{\beta}[0], \widetilde{\beta}[0] + \widetilde{\beta}[1], \ldots, \sum_{j=0}^{d(\widetilde{\beta})\cdot L - 1} \widetilde{\beta}[j]\right)_\circlearrowright \qquad i \in [p(\widetilde{\beta})]\,.
\end{equation}
Above, $d(\widetilde{\beta})$ is the smallest positive integer such that $ \left(\sum_{j=0}^{L-1} \widetilde{\beta}[j]\right) \cdot d(\widetilde{\beta})$ is divisible by $K$, and $p(\widetilde{\beta}) = K / d(\widetilde{\beta})$.
\end{defn}
The remainder of this subsection (i.e., Section~\ref{sec:lempelslift}) abbreviates functions of $\widetilde{\beta}$ given above by dropping it as an argument. For example, $p$ is written rather than $p(\widetilde{\beta})$.

Observe that $\widetilde{\lambda}_i$ is a discrete integral of $\widetilde{\beta}$, with $i \in [p]$ the constant of integration. The number $d$ specified in Definition~\ref{def:lempelslift} is the smallest positive integer $q \in \{1,2,\dots, K\}$ such that integrating a cycle of $\widetilde{\beta}$ a total of $q$ times gives a cycle of $\widetilde{\lambda}_i$. Conversely, $\widetilde{\beta}$ is uniquely determined by a discrete derivative of $\widetilde{\lambda}_i$, which eliminates the constant of integration:
$$
\widetilde{\beta}^d = \left(\lambda_i[1] - \lambda_i[0], \lambda_i[2] - \lambda_i[1], \ldots, \lambda_i[d\cdot L - 1] - \lambda_i[d \cdot L - 2], \lambda_i[0] - \lambda_i[d\cdot L - 1]\right)_\circlearrowright \quad i \in [p]\,.
$$
Above, the power $d$ on the LHS denotes $\widetilde{\beta}$ is concatenated with itself $d$ times.

Note the sum of the lengths of the necklaces comprising Lempel's lift of $\widetilde{\beta}$ is $K \cdot L$. Other properties of the lift pertinent to constructing $P^{(K)}_L$-sequences are as follows.

\begin{lem}\label{lem:occurrences}
Suppose $\widetilde{\beta}$ is a length-$L$ necklace on the alphabet $[K]$, and $m$ is an integer satisfying $0 \leq m < L$. Suppose $\omega$ is a length-$m$ string on $[K]$, and $\xi_\ell$ is the length-$(m+1)$ string on $[K]$ given by
\begin{equation}\label{eq:xidef}
    \xi_\ell = \ell + \left(0, \omega[0], \omega[0] + \omega[1], \ldots, \sum_{j=0}^{m - 1} \omega[j]\right) \qquad \ell \in [K]\,.
\end{equation}
Then $\omega$ occurs $t$ times as a substring of $\widetilde{\beta}$ if and only if $\xi_\ell$ occurs $t$ times as a substring of the necklaces comprising Lempel's lift of $\widetilde{\beta}$. When $m=0$, $\omega$ is the length-$0$ string occurring as a substring at every character of $\widetilde{\beta}$.
\end{lem}
\begin{proof}
Start constructing a given $\widetilde{\lambda}_i$ by integrating $\widetilde{\beta}$ from its character index $0$ up to character index $w < L$. If $\omega$ occurs as a substring of $\widetilde{\beta}$ at index $w$, it follows from \eqref{eq:lambdadef} and \eqref{eq:xidef} that $\xi_y$ occurs as a substring of $\widetilde{\lambda}_i$ at its character index $w - 1$ for some $y \in [K]$, and vice versa. For $d > 1$, continue integrating $\widetilde{\beta}$ past its character index $w$ for another $L$ characters to encounter $\omega$ again. This time, how $p$ is defined in terms of the sum of $\widetilde{\beta}$'s characters implies $\omega$'s presence as a substring of $\widetilde{\beta}$ at index $w$ is a necessary and sufficient condition for $\xi_{y + p}$'s presence as a substring of $\widetilde{\lambda}_i$ at its character index $w - 1 + L$. More generally, $\omega$ occurs as a substring of $\widetilde{\beta}$ at its character index $w$ if and only if $\xi_{y + qp}$ occurs as a substring of $\widetilde{\lambda}_i$ at its character index $w - 1 + qL$ for $q \in [d]$, and all occurrences of $\xi_\ell$ in Lempel's lift of $\widetilde{\beta}$ for which the difference between $\ell$ and $y$ is divisible by $p$ are in $\widetilde{\lambda}_i$. An occurrence of $\xi_\ell$ at any other value of $\ell$ is easily seen from \eqref{eq:lambdadef} to be at a corresponding character index $w - 1 + qL$ of $\widetilde{\lambda}_j$ for particular $j \in [p]\setminus \{i\}$ and $q \in [d]$. So there is an invertible map from the set of distinct occurrences of $\omega$ as a substring of $\widetilde{\beta}$ into the set of distinct occurrences of $\xi_\ell$ as a substring of Lempel's lift of $\widetilde{\beta}$ for $\ell \in [K]$, giving the lemma.
\end{proof}
\begin{lem}\label{lem:ppreserve}
The number of occurrences of a given length-$m$ string on the alphabet $[K]$ for $m \leq L$ as a substring in the family of necklaces comprising Lempel's lift of a $P^{(K)}_L$-sequence on $[K]$ is $\lfloor L / K^{m-1} \rfloor$ or $\lceil L / K^{m-1} \rceil$.
\end{lem}
\begin{proof}
By definition, any length-$m$ string $\omega$ on $[K]$ has $\lfloor L / K^m \rfloor$ or $\lceil L / K^m \rceil$ occurrences as a substring of a $P^{(K)}_L$-sequence on $[K]$. By Lemma~\ref{lem:occurrences}, for $\ell \in [K]$, the length-$(m+1)$ string $\xi_\ell$ as defined in \eqref{eq:xidef} thus occurs $\lfloor L / K^m \rfloor$ or $\lceil L / K^m \rceil$ times as a substring in Lempel's lift of a $P^{(K)}_L$-sequence on $[K]$. Since any length-$m$ string on $[K]$ is a prefix of $K$ length-$(m+1)$ strings on $[K]$, multiply each of $\lfloor L / K^m \rfloor$ and $\lceil L / K^m \rceil$ by $K$ to arrive at the lemma.
\end{proof}

\subsection{Algorithm and analysis}\label{sec:algproof}

In this subsection (Section~\ref{sec:algproof}), $\widetilde{\alpha}$ is reserved to denote a $P^{(K)}_L$-sequence. Moreover, when a function from Definition~\ref{def:lempelslift} is invoked, and it has $\widetilde{\alpha}$ as an argument, that function is abbreviated by dropping the $\widetilde{\alpha}$. For example, $p$ now refers to $p(\widetilde{\alpha})$.

Lemma~\ref{lem:ppreserve} suggests a way to obtain a $P^{(K)}_{K\cdot L}$-sequence from a $P^{(K)}_{L}$-sequence $\widetilde{\alpha}$: join the necklaces in Lempel's lift of $\widetilde{\alpha}$ strategically to ensure the numbers of occurrences of specific strings as substrings do not violate the parameters of Definition~\ref{def:p}. Below, the procedure \hyperref[alg:join]{\textsc{LiftAndJoin}} includes an explicit prescription, and Theorem~\ref{thm:join} proves it works. They are preceded by a requisite lemma extending the discussion of cycle joining from Section~\ref{sec:notation}.

\begin{lem}
\label{lem:smalljoin}
Consider two necklaces $\widetilde{\beta}$ and $\widetilde{\gamma}$ on the alphabet $[K]$, and suppose the length-$(m-1)$ string $\omega$ is a substring of each of them. For every $k \in [K]$, suppose further that no length-$m$ string $(\omega, k)$ is a substring of each of $\widetilde{\beta}$ and $\widetilde{\gamma}$, and no length-$m$ string $(k, \omega)$ is a substring of each of $\widetilde{\beta}$ and $\widetilde{\gamma}$. Finally, suppose every length-$(m+1)$ string on $[K]$ occurs either zero times or one time as a substring of the family $\{\widetilde{\beta}, \widetilde{\gamma}\}$. Then
\begin{enumerate}
\item every length-$(m+1)$ string on $[K]$ occurs either zero times or one time as a substring of the necklace $\widetilde{\sigma}$ formed by joining $\widetilde{\beta}$ and $\widetilde{\gamma}$ at $\omega$, and
\item every length-$w$ string for $w \leq m$ occurs the same number of times as a substring of $\{\widetilde{\beta}, \widetilde{\gamma}\}$ as it does as a substring of $\widetilde{\sigma}$.
\end{enumerate}
\end{lem}

\begin{proof}
For $u, v, x, y \in [K]$, suppose the length-$(m-1)$ string $\omega$ occurs (1) in $\widetilde{\beta}$ as a substring of the length-$(m+1)$ string $(u, \omega, v)$, and (2) in $\widetilde{\gamma}$ as a substring of the length-$(m+1)$ string $(x, \omega, y)$. Join $\widetilde{\beta}$ and $\widetilde{\gamma}$ at these occurrences of $\omega$ to obtain the necklace $\widetilde{\sigma}$. The operation replaces $(u, \omega, v)$ and $(w, \omega, x)$ with $(u, \omega, y)$ and $(x, \omega, v)$ while affecting the occurrence frequencies of no other length-$(m+1)$ strings as substrings and no length-$w$ strings as substrings for $w \leq m$. But $(u, \omega, y)$ cannot occur elsewhere as a substring of $\widetilde{\sigma}$ because if it does, then either $(u, \omega)$ or $(\omega, y)$ is a substring of each of $\widetilde{\beta}$ and $\widetilde{\gamma}$, a contradiction. By a parallel argument, $(x, \omega, v)$ cannot occur elsewhere in $\widetilde{\sigma}$. The lemma follows.
\end{proof}

\begin{algorithm}[H]
 \begin{algorithmic}[1]\label{alg:join}\onehalfspacing
\Statex \texttt{//\,\,Returns the $P^{(K)}_{K\cdot L}$-sequence formed by joining the necklaces}
\Statex \texttt{//\,\,comprising Lempel's lift of an input $P^{(K)}_L$-sequence $\widetilde{\alpha}$ on the}
\Statex \texttt{//\,\,alphabet $[K]$, with $K$ a clarifying input. Here, $N := \lceil \log_K L \rceil$.}
\Procedure{LiftAndJoin}{$\widetilde{\alpha}$, $K$}
\State Construct Lempel's lift $\{\widetilde{\lambda}_i : i \in [p]\}$ of $\widetilde{\alpha}$.
\If{$p = 1$}\hfill\texttt{// Case 1\qquad\qquad\qquad\,\,\,}
  \State \Return $\widetilde{\lambda}_0$
\EndIf
\If{$\mathbf{1}_N$ is a substring of $\widetilde{\alpha}$}\hfill\texttt{// Case 2\qquad\qquad\qquad\,\,\,}
  \State Find $k \in [K]$ such that $\mathbf{k}_N^{++}$ is a substring of each of $\widetilde{\lambda}_0$ and $\widetilde{\lambda}_1$.
  \State Initialize $\widetilde{\sigma}$ to $\widetilde{\lambda}_0$.
  \For{$j := 1 \textbf{ to } p - 1$}
      \State Set $\widetilde{\sigma}$ to the result of joining $\widetilde{\sigma}$ and $\widetilde{\lambda}_j$ at $\mathbf{s}_N^{++}$ for $s = k + j - 1$.
  \EndFor
  \State \Return $\widetilde{\sigma}$
\EndIf
  \State Construct the join graph $G = (V, E)$ defined in Theorem~\ref{thm:join}.\hfill\texttt{// Case 3\qquad\qquad\qquad\,\,\,}
  \State Initialize $\widetilde{\sigma}$ to the necklace represented by an arbitrary vertex $v \in V$.
  \State \begin{varwidth}[t]{4.2in} Starting at $v$, perform a depth-first traversal of the connected component $G_C = (V_C, E_C)$ of $G$ for which $v \in V_C$, where at each vertex in $V_C$ reached by walking across a given edge in $E_C$, the necklace represented by that vertex is joined with $\widetilde{\sigma}$ at the string labeling that edge, and the result is assigned to $\widetilde{\sigma}$.\end{varwidth}\vspace{1mm}
  \If{$G_C = G$}\hfill\texttt{// Case 3a\qquad\qquad\qquad\,}
    \State \Return $\widetilde{\sigma}$
  \EndIf
  \State Find $k \in [K]$ such that $\mathbf{k}_{N-1}^{++}$ is a substring of each of $\widetilde{\sigma}$ and $\widetilde{\sigma} + 1$.\hfill\texttt{// Case 3b\qquad\qquad\qquad\,}
  \State Initialize $\widetilde{\zeta}$ to $\widetilde{\sigma}$.
  \For{$j := 1 \textbf{ to } p / |V_C| - 1$}
      \State Set $\widetilde{\zeta}$ to the result of joining $\widetilde{\zeta}$ and $\widetilde{\sigma} + j$ at $\mathbf{s}_{N-1}^{++}$ for $s = k + j - 1$.
  \EndFor
  \State \Return $\widetilde{\zeta}$
\EndProcedure
\end{algorithmic}
\end{algorithm}

\begin{thm}\label{thm:join}
Given a $P^{(K)}_{L}$-sequence $\widetilde{\alpha}$ on the alphabet $[K]$, suppose $N = \lceil \log_K L \rceil$. Consider Lempel's lift $\{\widetilde{\lambda}_i : i \in [p]\}$ of $\widetilde{\alpha}$, and define the join graph $G = (V, E)$ as an undirected graph with $p$ vertices such that
\begin{enumerate}
\item the vertex $v_i \in V$ represents $\widetilde{\lambda}_i$ for $i \in [p]$, and
\item an edge in $E$ labeled by a length-$N$ string of the form $(\mathbf{j}_{N-1}^{++}, k)$ or $(k, \mathbf{j}_{N-1}^{++})$ for $j, k \in [K]$ extends between vertex $v_\ell$ and vertex $v_r$ if and only if that string occurs as a substring of each of $\widetilde{\lambda}_\ell$ and $\widetilde{\lambda}_r$ for $\ell, r \in [p]$.
\end{enumerate}
Then the length-$KL$ necklace output by \hyperref[alg:join]{\textsc{LiftAndJoin}} with $\widetilde{\alpha}$ and $K$ as inputs is a $P^{(K)}_{K\cdot L}$-sequence.
\end{thm}
\begin{proof}
Follow the logic of the \hyperref[alg:join]{\textsc{LiftAndJoin}} pseudocode to prove it returns a $P^{(K)}_{K \cdot L}$-sequence. To start, line 2 constructs Lempel's lift of $\widetilde{\alpha}$, which is composed of $p$ necklaces that together have precisely the same number of occurrences of any length-$m$ string on $[K]$ as a substring that a $P^{(K)}_{K\cdot L}$-sequence does, according to Lemma~\ref{thm:join}. To join the necklaces, various cases are handled in order of increasing difficulty:
\begin{enumerate}
\item[Case 1:] (Lines 3-5) This is the most straightforward case, where Lempel's lift has precisely one necklace. By Lemma~\ref{lem:ppreserve} and by definition of a $P^{(K)}_{L}$-sequence, the sole necklace is a $P^{(K)}_{K \cdot L}$-sequence, and it is returned (Line 4).
\item[Case 2:] (Line 6-13) In this case, $p > 1$ and $\mathbf{1}_N$ is a substring of $\widetilde{\alpha}$ so that by Lemma~\ref{lem:occurrences}, ${\bf k}^{++}_{N+1}$ is a substring of $\widetilde{\lambda}_1$ for at least one $k \in [K]$. Consequently, ${\bf s}^{++}_{N}$ is a substring of each of $\widetilde{\lambda}_j$ and $\widetilde{\lambda}_{j - 1}$  for $j \in [p]\setminus\{0\}$ and $s = k + j - 1$. Progressively joining a necklace under construction with the $j$th member $\widetilde{\lambda}_j$ of Lempel's lift of $\widetilde{\alpha}$ at ${\bf s}^{++}_{N}$ for $s = k + j - 1$ (Lines 8-11) and $j$ running from $1$ to $p-1$ preserves occurrence frequencies of all strings on $[K]$ whose lengths do not exceed $N+1$. Since by Lemma~\ref{lem:ppreserve} a length-$(N+1)$ string occurs either once or never as a substring of Lempel's lift of $\widetilde{\alpha}$, a string whose length exceeds $N+1$ occurs either once or never as a substring of the joined necklace. So that joined necklace is a $P^{(K)}_{K \cdot L}$-sequence, and it is returned (Line 12). When $\widetilde{\alpha}$ is a de Bruijn sequence (i.e., for $L = K^N$), Case 2 is the $K$-ary extension \cite{ronse1984feedback, tuliani2001bruijn, alhakim2011recursive} of the original join prescription of the paper \cite{lempel1970homomorphism} by Lempel introducing his D-morphism.

\item[Case 3:] (Lines 14-25) Because a length-$N$ string on $[K]$ need not occur as a substring of $\widetilde{\alpha}$, $\mathbf{1}_N$ may not be a substring of $\widetilde{\alpha}$. This bars the availability of Lempel's join of Case 2. \hyperref[alg:join]{\textsc{LiftAndJoin}} then looks for the closest alternative. By definition of a $P_{L}^{(K)}$-sequence, ${\bf 1}_{N-1}$ is necessarily a substring of $\widetilde{\alpha}$, and so by Lemma~\ref{lem:occurrences}, $\mathbf{j}_{N-1}^{++}$ is a substring of each necklace in Lempel's lift of $\widetilde{\alpha}$ for some $j \in [K]$. So Line 14 assembles the graph $G$ encoding all possible joins at strings of the form $({\bf j}_{N-1}^{++}, k)$ or $(k, {\bf j}_{N-1}^{++})$ for $j, k \in [K]$. Consider any connected component $G_C = (V_C, E_C)$ of $G$. A depth-first traversal of $G_C$ prescribes a sequence of joins, which are performed to obtain a single necklace $\widetilde{\sigma}$ (Line 16). Two cases are then considered.
\begin{enumerate}
    \item[Case 3a:] (Lines 17-19) In this case, there is just one connected component of $G$. Since each join was performed at a length-$N$ string, by an argument parallel to that of Case 2, $\widetilde{\sigma}$ is a $P^{(K)}_{K \cdot L}$-sequence, and it is returned (Line 18).
    \item[Case 3b:] (Lines 20-25) If there are multiple connected components of $G$, by symmetry, $G_C$ is related to any other connected component by translation modulo $K$. More precisely, applying $v_k \to v_{k+j}$ to each vertex $v_k \in V_C$, $e_{k \ell} \to e_{k + j, \ell + j}$ to each edge $e_{k \ell} \in E_C$ extending between $v_k \in V_C$ and $v_\ell \in V_C$, and $\epsilon_{k \ell} \to \epsilon_{k + j, \ell + j} + j$ to each edge label $\epsilon_{k\ell}$ corresponding to $e_{k\ell} \in E_C$ gives a different connected component, where $j \in [p / |V_C|]$ and addition operations are performed modulo $K$. It follows that for every $j \in [p / |V_C|]$, $\widetilde{\sigma} + j$ gives the result of a sequence of joins prescribed by a different connected component of $G$. Because each join was performed at a length-$N$ string, the necklaces $\{ \widetilde{\sigma} + j : j \in [p / |V_C|]\}$ together have the same occurrence frequency of any length-$m$ string on $[K]$ as does Lempel's lift of $\widetilde{\alpha}$ for $m \leq N + 1$. That occurrence frequency is $0$ or $1$ for $m = N + 1$, as it therefore also is for $m > N + 1$. Because possible joins at strings of the form $({\bf s}_{N-1}^{++}, k)$ or $(k, {\bf s}_{N-1}^{++})$ for $s, k \in [K]$ were exhausted by prior joins, Lemma \ref{lem:smalljoin} guarantees that joins of the $\{ \widetilde{\sigma} + j : j \in [p / |V_C|]\}$ at strings of the form ${\bf s}_{N-1}^{++}$ for $s \in [K]$ preserve the occurrence frequency of any length-$m$ string on $[K]$ for $m \leq N$ while ensuring that when $m > N$, the occurrence frequency of a length-$m$ string remains either $0$ or $1$. So when all necklaces in $\{ \widetilde{\sigma} + j : j \in [p / |V_C|]\}$ are joined as on Lines 21-24, the result is a $P_{K \cdot L}^{(K)}$-sequence, and it is returned (Line 25). Note the joins are performed in exact analogy to those of Case 2.
\end{enumerate}
\end{enumerate}
The output of \hyperref[alg:join]{\textsc{LiftAndJoin}} is thus a $P^{(K)}_{K\cdot L}$-sequence.
\end{proof}

Repeated application of \hyperref[alg:join]{\textsc{LiftAndJoin}} on a $P_L^{(K)}$-sequence $\widetilde{\alpha}$ outputs a $P_L^{(K)}$-sequence whose length multiplies the length of $\widetilde{\alpha}$ by a power of $K$. But this operation alone does not afford the expressive capacity to build up a $P_L^{(K)}$-sequence of arbitrary length starting from an $\widetilde{\alpha}$ of length less than $K$, in the same way that an arbitrary positive integer cannot be written as a power of $K$ times a positive integer less than $K$. A mechanic for extending the length of $\widetilde{\alpha}$ by up to $K - 1$ between applications of \hyperref[alg:join]{\textsc{LiftAndJoin}} is required, where the length of the extension is determined by an appropriate digit from the base-$K$ representation of $L$. The mechanic used in the iterative procedure \hyperref[alg:main]{\textsc{GeneratePKL}} below, which outputs a $P_L^{(K)}$-sequence for any combination of $K\geq 2$ and $L \geq 1$, extends a given longest run of a nonzero character by a single character. Theorem~\ref{thm:main} proves this approach works.

\begin{algorithm}[H]
 \begin{algorithmic}[1]\label{alg:main}\onehalfspacing
\Statex \texttt{//\,\,Returns a $P^{(K)}_L$-sequence on the alphabet $[K]$ given $K \geq 2$ and }
\Statex \texttt{//\,\,$L \geq 1$ as inputs. Here, $N := \lceil \log_K L \rceil$.}
\Procedure{GeneratePKL}{$K$, $L$}
 \State \begin{varwidth}[t]{4.2in}Compute the digits $\{d_i\}$ of $L$ in its base-$K$ representation as specified by $L = \sum_{i=0}^{N - 1} d_i K^{N - i - 1}$.\end{varwidth}\vspace{1mm}
 \State Initialize the necklace $\widetilde{\alpha}$ to $\widetilde{{\bf 1}}^{++}_{d_0}$.
  \For{$j := 1 \textbf{ to } N - 1$}
  \State Set $\widetilde{\alpha}$ to \hyperref[alg:join]{\textsc{LiftAndJoin}}($\widetilde{\alpha}$, $K$). 
  \State\begin{varwidth}[t]{4in}Set $\widetilde{\alpha}$ to the extension of $\widetilde{\alpha}$ by $d_j$ characters as obtained by replacing a substring ${\bf k}_{j - 1}$ with ${\bf k}_{j}$ for every $k \in \{1, \ldots, d_j\}$. \end{varwidth}\vspace{1mm}
  \EndFor
  \State \Return $\widetilde{\alpha}$
  \EndProcedure
 \end{algorithmic}
\end{algorithm}

\begin{thm}\label{thm:main}
\hyperref[alg:main]{\textsc{GeneratePKL}}$(K, L)$ outputs a $P_L^{(K)}$-sequence for any combination of $K \geq 2$ and $L \geq 1$.
\end{thm}
\begin{proof}
Use the notation $\widetilde{\alpha}_0$ to denote the value of $\widetilde{\alpha}$ after Line 3 of \hyperref[alg:main]{\textsc{GeneratePKL}} is executed and the notation $\widetilde{\alpha}_{j}$ to denote the value of $\widetilde{\alpha}$ after step $j$ of the \textbf{for} loop of \hyperref[alg:main]{\textsc{GeneratePKL}}. Prove the theorem by induction, showing that if $\widetilde{\alpha}_{j-1}$ is a $P_{L_{j-1}}^{(K)}$-sequence of length $L_{j-1}$, and $\mathbf{0}_{m}$ occurs $\lfloor L_{j-1} / K^m \rfloor$ times as a substring of $\widetilde{\alpha}_{j-1}$ for all $m \leq L_{j-1}$, then $\widetilde{\alpha}_{j}$ is a $P_{L_{j}}^{(K)}$-sequence of length $L_{j} = K \cdot L_{j - 1} + d_j$, and $\mathbf{0}_{n}$ occurs $\lfloor L_{j} / K^n \rfloor$ times as a substring of $\widetilde{\alpha}_{j}$ for all $n \leq L_{j}$. The base case for the induction holds: $\widetilde{\alpha}_{0}$, as initialized on Line 3, is the $P_{L_{0}}^{(K)}$-sequence $\widetilde{\mathbf{1}}_{d_0}^{++}$ of length $L_{0} = d_0$, in which $\mathbf{0}_{m}$ occurs as a substring $\lfloor d_0 / K^m \rfloor = 0$ times for $1 \leq m \leq d_0$ and $\lfloor d_0 / K^m \rfloor = d_0$ times for $m = 0$. Now suppose that $\widetilde{\alpha}_{j-1}$ is a $P_{L_{j-1}}^{(K)}$-sequence of length $L_{j-1}$, and $\mathbf{0}_{m}$ occurs $\lfloor L_{j-1} / K^m \rfloor$ times as a substring of $\widetilde{\alpha}_{j-1}$ for all $m \leq L_{j-1}$. Then for every $k \in [K]$, $\mathbf{k}_{m+1}$ occurs $\lfloor L_{j-1} / K^m \rfloor = \lfloor (K \cdot L_{j-1}) / K^{m+1} \rfloor$ times as a substring of \hyperref[alg:join]{\textsc{LiftAndJoin}}$(\widetilde{\alpha}_{j-1}, K)$, obtained on Line 5. This follows from
\begin{enumerate}
    \item Lemma~\ref{lem:occurrences}, which says there are $t$ occurrences of $\mathbf{0}_m$ as a substring of a necklace if and only if there are $t$ occurrences of $\mathbf{k}_{m+1}$ as a substring in Lempel's lift of that necklace, and
    \item how all joins of necklaces in Lempel's lift prescribed by \hyperref[alg:join]{\textsc{LiftAndJoin}}, including those permitted by Lemma~\ref{lem:smalljoin}, do not affect occurrences of substrings of the form $\mathbf{k}_{m+1}$.  
\end{enumerate}
The extension performed on Line 6 increases the number of occurrences of $\mathbf{k}_m$, for $k = 1,2,\dots, d_j,$ from $\lfloor (K \cdot L_{j-1}) / K^{m+1} \rfloor$ to $\lceil (K \cdot L_{j-1}) / K^{m+1} \rceil$ without affecting the numbers of occurrences of any other length-$m$ strings as substrings for $m \leq j$. The longest string of 0s is never extended, and the number of occurrences of $\mathbf{0}_n$ remains $\lfloor (K \cdot L_{j-1}) / K^{n+1} \rfloor$ for all $n \leq L_j$. So the resulting necklace $\widetilde{\alpha}_{j}$ is a $P_{L_{j}}^{(K)}$-sequence of length $L_{j} = K \cdot L_{j - 1} + d_j$, and $\mathbf{0}_{n}$ occurs $\lfloor L_{j} / K^n \rfloor$ times as a substring of $\widetilde{\alpha}_{j}$ for all $n \leq L_{j}$.

The \textbf{for} loop thus encodes the recursion \begin{equation}\label{eq:Lrecurse}
L_j = K \cdot L_{j - 1} + d_j\quad j \in [N-1]\setminus\{0\}\,
\end{equation}
with initial condition $L_0 = d_0$. The formula $L = L_{N-1} = \sum_{i=0}^{N-1} d_i K^{N - i - 1}$ follows, concluding the proof.
\end{proof}

In the binary case, \hyperref[alg:main]{\textsc{GeneratePKL}} and the joins it requires of its subroutine \hyperref[alg:join]{\textsc{LiftAndJoin}} collapse to a particularly simple algorithm, which is given in the procedure \hyperref[alg:binary]{\textsc{GenerateP2L}} below.

\begin{algorithm}[H]
 \begin{algorithmic}[1]\label{alg:binary}\onehalfspacing
\Statex \texttt{//\,\,Returns a $P^{(2)}_L$-sequence on the alphabet $\{0, 1\}$ given $L \geq 1$ as}
\Statex \texttt{//\,\,an input. Here, $N := \lceil \log_2 L \rceil$.}
\Procedure{GenerateP2L}{$L$}
\State \begin{varwidth}[t]{4.2in}Compute the digits $\{d_i\}$ of $L$ in its binary representation as specified by $L = \sum_{i=0}^{N - 1} d_i 2^{N - i - 1}$.\end{varwidth}\vspace{1mm}
 \State Initialize the necklace $\widetilde{\alpha}$ to the single character $1$.
  \For{$j := 1 \textbf{ to } N - 1$}
  \State Construct Lempel's lift $\{\widetilde{\lambda}_i : i \in [p]\}$ of $\widetilde{\alpha}$.
  \If{p = 1}
   \Statex \qquad\qquad\,\,\,\texttt{//\,\,$\widetilde{\alpha}$ has an odd number of $1$s.}
    \State Set $\widetilde{\alpha}$ to $\widetilde{\lambda}_0$.
  \ElsIf{$\mathbf{1}_{j-1}$ is a substring of $\widetilde{\alpha}$}
    \State Set $\widetilde{\alpha}$ to the result of joining $\widetilde{\lambda}_0$ and $\widetilde{\lambda}_1$ at $\mathbf{0}_{j-1}^{++}$.
  \ElsIf{$\widetilde{\lambda}_0$ and $\widetilde{\lambda}_1$ can be joined at $(\mathbf{0}_{j - 2}^{++}, k)$ or $(k,  \mathbf{0}_{j - 2}^{++})$ for $k \in [2]$}
    \State \begin{varwidth}[t]{4in}Set $\widetilde{\alpha}$ to the result of joining $\widetilde{\lambda}_0$ and $\widetilde{\lambda}_1$ at $(\mathbf{0}_{j-2}^{++}, k)$ or $(k, \mathbf{0}_{j-2}^{++})$ for $k \in [2]$.\end{varwidth}\vspace{1mm}
  \Else
     \State Set $\widetilde{\alpha}$ to the result of joining $\widetilde{\lambda}_0$ and $\widetilde{\lambda}_1$ at $\mathbf{0}_{j-2}^{++}$.
  \EndIf
  \EndFor
  \State \Return $\widetilde{\alpha}$
  \EndProcedure
 \end{algorithmic}
\end{algorithm}

Below is the final theorem of this paper, which proves complexity results.
\begin{thm}
\hyperref[alg:main]{\textsc{GeneratePKL}} outputs a $P^{(K)}_L$-sequence in $O(L)$ time using $O(L \log K)$ space.
\end{thm}
\begin{proof}
The space required by \hyperref[alg:main]{\textsc{GeneratePKL}} is dominated by storage of the final $P^{(K)}_L$-sequence itself, which is $O(L \log K)$.

To see why the algorithm takes $O(L)$ time, consider first the case $L < K$. \hyperref[alg:main]{\textsc{GeneratePKL}} then initializes $\widetilde{\alpha}$ to the positive integers in order up to and including $L$ (Line 3), which scales as $L$. It subsequently skips the \textbf{for} loop and returns $\widetilde{\alpha}$.

Now consider the opposite case $L \geq K$. Expressing $L$ in base $K$ (Line 2) scales as $\log_K L$, and initializing $\widetilde{\alpha}$ (Line 3) scales as $K$. Focus on Line 5's call of \hyperref[alg:join]{\textsc{LiftAndJoin}} at step $j$ of the \textbf{for} loop, where the length-$L_{j-1}$ necklace $\widetilde{\alpha}_{j-1}$ is passed to \hyperref[alg:join]{\textsc{LiftAndJoin}} in the notation of Theorem~\ref{thm:main}'s proof. Constructing Lempel's lift of $\widetilde{\alpha}_{j-1}$ (Line 2 of \hyperref[alg:join]{\textsc{LiftAndJoin}}) scales as $K \cdot L_{j-1}$, the total length of the necklaces constructed. Addressing Case 1 (Lines 3-5) takes constant time. Addressing Case 2 (Lines 6-13) involves searching $\widetilde{\alpha}_{j-1}$ for $\mathbf{1}_N$, which scales as $L_{j-1}$, and successively joining the necklaces comprising Lempel's lift, which scales as $K \cdot L_{j-1}$ if implemented as, e.g., a sequence of rotations and concatenations in which indexes of join substrings are tracked. Addressing Case 3 in its entirety (Lines 14-24) involves (1) constructing the join graph $G$, which is dominated by the $K \cdot L_{j-1}$ scaling of searching Lempel's lift for strings of the form $\mathbf{s}_{N-1}^{++}$ for $s \in [K]$, (2) performing a depth-first traversal of a connected component of the join graph, which takes time linear in a number of at most $K$ vertices, and (3) joining necklaces, which also scales as $K \cdot L_{j-1}$. So \hyperref[alg:join]{\textsc{LiftAndJoin}} is dominated by a $K \cdot L_{j-1}$ scaling at step $j$ of the \textbf{for} loop of \hyperref[alg:main]{\textsc{GeneratePKL}}. Refocusing on \hyperref[alg:main]{\textsc{GeneratePKL}}, extending \hyperref[alg:join]{\textsc{LiftAndJoin}}$(\widetilde{\alpha}_{j-1}, K)$ (Line 6) involves searching for longest runs of the same character and inserting characters as necessary, scaling as $K \cdot L_{j-1}$ if performed in one pass through the necklace. Therefore, step $j$ of the \textbf{for} loop scales as $K \cdot L_{j-1}$, and from the recursion \eqref{eq:Lrecurse}, executing all iterations of the \textbf{for} loop scales as $L$. The time taken by the \textbf{for} loop dominates that of Lines 2 and 3. So the overall scaling is $L$ for the two cases $L \geq K$ and $L < K$, and \hyperref[alg:main]{\textsc{GeneratePKL}} takes $O(L)$ time.
\end{proof}

\section{Discussion}

In this paper, we have introduced $P^{(K)}_L$-sequences as arbitrary-length analogs to de Bruijn sequences. We have shown by explicit construction that a $P^{(K)}_L$-sequence exists for any combination of $K \geq 2$ and $L \geq 1$, giving an $O(L)$-time, $O(L \log K)$-space algorithm extending Lempel's recursive construction of a binary de Bruijn sequence. An implementation of the algorithm in Python is available at \url{https://github.com/nelloreward/pkl}.

\begin{table}[!b]
\begin{center}
\begin{tabular}{@{}lllll@{}}
\toprule
$L$  & Number of distinct $P^{(2)}_L$-sequences & & $L$  & Number of distinct $P^{(2)}_L$-sequences \\ \midrule
1  & 2                                                         &  & 17 & 32                                                        \\
2  & 1                                                         &  & 18 & 36                                                        \\
3  & 2                                                         &  & 19 & 68                                                        \\
4  & 1                                                         &  & 20 & 57                                                        \\
5  & 2                                                         &  & 21 & 138                                                       \\
6  & 3                                                         &  & 22 & 123                                                       \\
7  & 4                                                         &  & 23 & 252                                                       \\
8  & 2                                                         &  & 24 & 378                                                       \\
9  & 4                                                         &  & 25 & 504                                                       \\
10 & 3                                                         &  & 26 & 420                                                       \\
11 & 6                                                         &  & 27 & 1296                                                      \\
12 & 9                                                         &  & 28 & 1520                                                      \\
13 & 12                                                        &  & 29 & 2176                                                      \\
14 & 20                                                        &  & 30 & 2816                                                      \\
15 & 32                                                        &  & 31 & 4096                                                      \\
16 & 16                                                        &  & 32 & 2048                                                      \\
\bottomrule
\end{tabular}
\end{center}
\caption{Numbers of distinct $P^{(2)}_L$ sequences on $\mathcal{A}$ for various values of $L$.}
\label{tab:search}
\end{table}

We conclude with several open questions suggested by our work:
\begin{enumerate}
    \item What is the number of distinct $P^{(K)}_L$-sequences on $\mathcal{A}$ for every possible combination of $K$ and $L$? As Gabric, Holub, and Shallit did in \cite{gabric2019generalized,gabric2021maximal} for generalized de Bruijn sequences, we have counted $P^{(2)}_L$-sequences for $L$ up to $32$ by exhaustive search. Table~\ref{tab:search} displays our results, which can be reproduced using code at \url{https://github.com/nelloreward/pkl}. Note the counts do not increase monotonically with $L$.
    \item Can the algorithm for $P^{(K)}_L$-sequence generation presented here or a variant be encoded in a shift rule? This would reduce the space it requires, perhaps at the expense of performance. An obstacle to deriving a shift rule from the algorithm that works for all values of $L$ at a given alphabet size $K$ is that it would have to account for cases like those of \hyperref[alg:join]{\textsc{LiftAndJoin}}. See \cite{mykkeltveit1979cycle, siu1980generation,alhakim2017stretching} for work along the lines of mathematically unrolling Lempel's recursion and generalizations.
    \item Are there elegant constructions of $P^{(K)}_L$-sequences for any possible combination of $K$ and $L$ that extend constructions of de Bruijn sequences besides Lempel's recursive construction? There is a considerable body of literature on constructing universal cycles. (See, e.g., \cite{hurlbert1996equivalence,johnson2009universal,sawada2014lexicographically,sawada2016generalizing,wong2017new,gabric2019successor,sawada2020efficient}). Introduced by Chung, Diaconis, and Graham in \cite{chung1992universal}, a \textit{universal cycle} is a length-$L$ necklace in which every string in a size-$L$ set $S$ of length-$m$ strings occurs as a substring. It is possible a set $S$ curated to ensure the universal cycle is a $P^{(K)}_L$-sequence is compatible with existing universal cycle constructions or extensions.
    \item Can an efficiently decodable $P^{(K)}_L$-sequence be constructed for any possible combination of $K$ and $L$? Toward answering this question, it may be worth further investigating the efficient decoding of Lempel's recursive construction of a de Bruijn sequence. (See \cite{paterson1995storage} for the binary case and \cite{tuliani2001bruijn} for the $K$-ary case.) Other efficiently decodable constructions of de Bruijn sequences are given in \cite{kociumaka2016efficient, sawada2017practical}.
    \item What other properties that can be exhibited by a necklace are preserved under Lempel's D-morphism, and how can they be exploited to recursively construct other useful sequences? While this work was being prepared, Mitchell and Wild posted \cite{mitchell2021constructing} to arXiv, which shows binary orientable sequences can be constructed recursively using Lempel's D-morphism. An \textit{orientable sequence} is a necklace $\widetilde{\nu}$ for which each length-$n$ substring has precisely one occurrence in precisely one of $\widetilde{\nu}$ and the reverse of $\widetilde{\nu}$ \cite{burns1992coding,dai1993orientable}. It is perhaps unsurprising that Lempel's D-morphism is so versatile and that orientability, $P^{(K)}_L$-sequence composition, and efficient decodability can be preserved by its inverse. Indeed, Lempel's D-morphism is a kind of derivative, and just as derivatives of special functions have special properties (e.g., the period of a sine function is the same as the period of its derivative), it appears derivatives of special necklaces have special properties.
\end{enumerate}

\begin{appendices}
\section*{Appendix: A brief history of de Bruijn sequences} \label{sec:appendix}

The earliest known recorded de Bruijn sequence is the Sanskrit sutra yam$\bar{\text{a}}$t$\bar{\text{a}}$r$\bar{\text{a}}$jabh$\bar{\text{a}}$nasalag$\bar{\text{a}}$m, a mnemonic encoding all possible length-$3$ combinations of short and long vowels \cite{kak2000yamatarajabhanasalagam}. Historians have had some trouble placing when it was first conceived, but it may be over 2,500 years old \cite{brown1869sanskrit}, having appeared in work by the ancient Indian scholar P$\bar{\text{a}}\underdot{\text{n}}$ini (dates of birth and death unavailable). Little is known of P$\bar{\text{a}}\underdot{\text{n}}$ini beyond his foundational work \textit{A$\underdot{\text{s}}\underdot{\text{t}}\bar{\text{a}}$dhy$\bar{\text{a}}\bar{\text{\i}}$} codifying Sanskrit grammatical structure \cite{vasu1897ashtadhyayi}.

The question of whether and how many binary de Bruijn sequences of every order exist was first posed by A.~de Rivi\`ere (dates of birth and death unavailable) as problem 48 of \cite{laisant1894intermediaire} in 1894. That same year, in response to the problem, the number of binary de Bruijn sequences of every order was counted in \cite{saint1894solution} by Camille Flye Sainte-Marie (1834-1926), a member of the Mathematical Society of France who was affiliated with the French military throughout his career \cite{camillebio}. His work was quickly forgotten.

Monroe H. Martin (1907-2007) was first to prove the existence of de Bruijn sequences of any order for any alphabet size in his 1934 paper \cite{martin1934problem} by explicit construction, shortly before arriving at the University of Maryland, where he spent the rest of his eminent career. Without knowing of Sainte-Marie's work, Dick de Bruijn (1918-2012) \cite{debruijnbio} also counted the number of binary de Bruijn sequences of every order in his 1946 work \cite{de1946combinatorial}. Tatyana van Aardenne-Ehrenfest and de Bruijn were first to prove the formula for the number of de Bruijn sequences of any order for any alphabet size in their 1951 paper \cite{de1951circuits}. It is notable that after receiving her PhD from the University of Leiden in 1931, van Ardenne-Ehrenfest (1905-1984) made this and further significant contributions to the mathematics of sequences despite never holding paid employment as a mathematician and working as a homemaker \cite{de1985memoriam}.

Sainte-Marie's work was ultimately rediscovered by the well-known MIT combinatorialist Richard Peter Stanley (1944-) \cite{hersh2016mathematical}, who brought it to the attention of de Bruijn, and in 1975, de Bruijn issued an acknowledgement \cite{de1975acknowledgement} of the work. In this acknowledgement, de Bruijn noted that as early as 1897, Willem Mantel (dates of birth and death unavailable) showed how to construct de Bruijn sequences of any order for any alphabet size that is prime \cite{mantel1897resten}, also in response to A.~de Rivi\`ere's problem.
\end{appendices}

\section*{Acknowledgements}
We thank Arie Israel for his helpful comments as we prepared this manuscript. AN thanks the Oden Institute for Computational Engineering \& Sciences at the University of Texas at Austin for hosting him as a visiting scholar as part of the J. Tinsley Oden Faculty Fellowship Research Program when this work was initiated.

\bibliographystyle{ieeetr}
\bibliography{references}  %%% Uncomment this line and comment out the ``thebibliography'' section below to use the external .bib file (using bibtex) .

\begin{thebibliography}{10}

\bibitem{saint1894solution}
C.~F. Saint-Marie, ``Solution to question nr. 48,'' {\em L'Interm\'ediaire des
  Math\'ematiciens}, 1894.

\bibitem{de1946combinatorial}
N.~G. De~Bruijn, ``A combinatorial problem,'' in {\em Proc. Koninklijke
  Nederlandse Academie van Wetenschappen}, vol.~49, pp.~758--764, 1946.

\bibitem{de1951circuits}
N.~G. de~Bruijn and T.~van Aardenne-Ehrenfest, ``Circuits and trees in oriented
  linear graphs,'' {\em Simon Stevin}, vol.~28, pp.~203--217, 1951.

\bibitem{de1975acknowledgement}
N.~G. de~Bruijn, ``Acknowledgement of priority to {C.} {Flye} {Sainte-Marie} on
  the counting of circular arrangements of $2^n$ zeros and ones that show each
  n-letter word exactly once,'' {\em EUT report. WSK, Dept. of Mathematics and
  Computing Science}, vol.~75, 1975.

\bibitem{yoeli1963counting}
M.~Yoeli, ``Counting with nonlinear binary feedback shift registers,'' {\em
  IEEE Transactions on Electronic Computers}, pp.~357--361, 1963.

\bibitem{radchenko1959shifting}
A.~Radchenko and V.~Filippov, ``Shifting registers with logical feedback and
  their use as counting and coding devices,'' {\em Automation and Remote
  Control (English translation of Soviet Journal Automatika i Telemekhanika)},
  vol.~20, pp.~1467--1473, nov 1959.

\bibitem{radchenko1958code}
A.~Radchenko, {\em Code Rings and Their Use in Contactless Coding Devices}.
\newblock PhD thesis, University of Leningrad, USSR, 1958.

\bibitem{yoeli1961nonlinear}
M.~Yoeli, {\em Nonlinear Feedback Shift Registers}.
\newblock International Business Machines Corporation, Development
  Laboratories, Data Systems Division, 1961.

\bibitem{yoeli1962binary}
M.~Yoeli, ``Binary ring sequences,'' {\em The American Mathematical Monthly},
  vol.~69, no.~9, pp.~852--855, 1962.

\bibitem{bryant1962counting}
P.~Bryant, F.~Heath, and R.~Killick, ``Counting with feedback shift registers
  by means of a jump technique,'' {\em IRE Transactions on Electronic
  Computers}, pp.~285--286, 1962.

\bibitem{heath1961chain}
F.~Heath and M.~Gribble, ``Chain codes and their electronic applications,''
  {\em Proceedings of the IEE-Part C: Monographs}, vol.~108, no.~13,
  pp.~50--57, 1961.

\bibitem{golomb1959cycles}
S.~W. Golomb, L.~R. Welch, and R.~M. Goldstein, ``Cycles from nonlinear shift
  registers,'' tech. rep., Jet Propulsion Lab, Pasadena, CA, 1959.

\bibitem{etzion1986algorithma}
T.~Etzion, ``An algorithm for generating shift-register cycles,'' {\em
  Theoretical computer science}, vol.~44, pp.~209--224, 1986.

\bibitem{lempel1971m}
A.~Lempel, ``m-ary closed sequences,'' {\em Journal of Combinatorial Theory,
  Series A}, vol.~10, no.~3, pp.~253--258, 1971.

\bibitem{hemmati1978algebraic}
F.~Hemmati and D.~J. Costello, ``An algebraic construction for q-ary shift
  register sequences,'' {\em IEEE Transactions on Computers}, vol.~100, no.~12,
  pp.~1192--1195, 1978.

\bibitem{golomb2017shift}
S.~W. Golomb, {\em Shift Register Sequences: Secure and Limited-Access Code
  Generators, Efficiency Code Generators, Prescribed Property Generators,
  Mathematical Models}.
\newblock World Scientific, 2017.

\bibitem{landsberg2000feedback}
M.~Landsberg, ``Feedback functions for generating cycles over a finite
  alphabet,'' {\em Discrete Mathematics}, vol.~219, no.~1-3, pp.~187--194,
  2000.

\bibitem{etzion1986algorithmb}
T.~Etzion, ``An algorithm for constructing m-ary de {Bruijn} sequences,'' {\em
  Journal of algorithms}, vol.~7, no.~3, pp.~331--340, 1986.

\bibitem{sawadasite}
``De {Bruijn} sequence and universal cycle constructions.''
  \url{https://debruijnsequence.org/}.
\newblock Accessed: 2021-07-24.

\bibitem{gundogan2021cut}
A.~G\"undo\v{g}an, B.~Cameron, and J.~Sawada, ``Cut-down de {Bruijn}
  sequences,'' 2021.
\newblock Unpublished manuscript.

\bibitem{gabric2019generalized}
D.~Gabric, {\v{S}}.~Holub, and J.~Shallit, ``Generalized de {Bruijn} words and
  the state complexity of conjugate sets,'' in {\em International Conference on
  Descriptional Complexity of Formal Systems}, pp.~137--146, Springer, 2019.

\bibitem{gabric2021maximal}
D.~Gabric, Štěpán Holub, and J.~Shallit, ``Maximal state complexity and
  generalized de {Bruijn} words,'' {\em Information and Computation},
  p.~104689, 2021.

\bibitem{martin1934problem}
M.~H. Martin, ``A problem in arrangements,'' {\em Bulletin of the American
  Mathematical Society}, vol.~40, no.~12, pp.~859--864, 1934.

\bibitem{eldert1958shifting}
C.~Eldert, H.~Gurk, H.~Gray, and M.~Rubinoff, ``Shifting counters,'' {\em
  Transactions of the American Institute of Electrical Engineers, Part I:
  Communication and Electronics}, vol.~77, no.~1, pp.~70--74, 1958.

\bibitem{fredricksen1982survey}
H.~Fredricksen, ``A survey of full length nonlinear shift register cycle
  algorithms,'' {\em SIAM review}, vol.~24, no.~2, pp.~195--221, 1982.

\bibitem{alhakim2021revisiting}
A.~Alhakim, E.~Sala, and J.~Sawada, ``Revisiting the prefer-same and
  prefer-opposite de {Bruijn} sequence constructions,'' {\em Theoretical
  Computer Science}, vol.~852, pp.~73--77, 2021.

\bibitem{alhakim2010simple}
A.~M. Alhakim, ``A simple combinatorial algorithm for de {Bruijn} sequences,''
  {\em The American Mathematical Monthly}, vol.~117, no.~8, pp.~728--732, 2010.

\bibitem{sawada2016surprisingly}
J.~Sawada, A.~Williams, and D.~Wong, ``A surprisingly simple de {Bruijn}
  sequence construction,'' {\em Discrete Mathematics}, vol.~339, no.~1,
  pp.~127--131, 2016.

\bibitem{huang1990new}
Y.~Huang, ``A new algorithm for the generation of binary de {Bruijn}
  sequences,'' {\em Journal of Algorithms}, vol.~11, no.~1, pp.~44--51, 1990.

\bibitem{etzion1984algorithms}
T.~Etzion and A.~Lempel, ``Algorithms for the generation of full-length
  shift-register sequences,'' {\em IEEE Transactions on Information Theory},
  vol.~30, no.~3, pp.~480--484, 1984.

\bibitem{fredricksen1975class}
H.~Fredricksen, ``A class of nonlinear de {Bruijn} cycles,'' {\em Journal of
  Combinatorial Theory, Series A}, vol.~19, no.~2, pp.~192--199, 1975.

\bibitem{sawada2017simple}
J.~Sawada, A.~Williams, and D.~Wong, ``A simple shift rule for k-ary de
  {Bruijn} sequences,'' {\em Discrete Mathematics}, vol.~340, no.~3,
  pp.~524--531, 2017.

\bibitem{amram2019efficient}
G.~Amram, Y.~Ashlagi, A.~Rubin, Y.~Svoray, M.~Schwartz, and G.~Weiss, ``An
  efficient shift rule for the prefer-max de {Bruijn} sequence,'' {\em Discrete
  Mathematics}, vol.~342, no.~1, pp.~226--232, 2019.

\bibitem{jansen1991efficient}
C.~J. Jansen, W.~G. Franx, and D.~E. Boekee, ``An efficient algorithm for the
  generation of debruijn cycles,'' {\em IEEE Transactions on Information
  Theory}, vol.~37, no.~5, pp.~1475--1478, 1991.

\bibitem{gabric2018framework}
D.~Gabric, J.~Sawada, A.~Williams, and D.~Wong, ``A framework for constructing
  de {Bruijn} sequences via simple successor rules,'' {\em Discrete
  Mathematics}, vol.~341, no.~11, pp.~2977--2987, 2018.

\bibitem{gabric2019successor}
D.~Gabric, J.~Sawada, A.~Williams, and D.~Wong, ``A successor rule framework
  for constructing k-ary de {Bruijn} sequences and universal cycles,'' {\em
  IEEE Transactions on Information Theory}, vol.~66, no.~1, pp.~679--687, 2019.

\bibitem{yang1992construction}
J.-H. Yang and Z.-D. Dai, ``Construction of m-ary de {Bruijn} sequences,'' in
  {\em International Workshop on the Theory and Application of Cryptographic
  Techniques}, pp.~357--363, Springer, 1992.

\bibitem{chang2019general}
Z.~Chang, M.~F. Ezerman, P.~Ke, and Q.~Wang, ``General criteria for successor
  rules to efficiently generate binary de {Bruijn} sequences,'' {\em arXiv
  preprint arXiv:1911.06670}, 2019.

\bibitem{zhu2021efficiently}
Y.~Zhu, Z.~Chang, M.~F. Ezerman, and Q.~Wang, ``An efficiently generated family
  of binary de {Bruijn} sequences,'' {\em Discrete Mathematics}, vol.~344,
  no.~6, p.~112368, 2021.

\bibitem{fredricksen1978necklaces}
H.~Fredricksen and J.~Maiorana, ``Necklaces of beads in k colors and k-ary de
  {Bruijn} sequences,'' {\em Discrete Mathematics}, vol.~23, no.~3,
  pp.~207--210, 1978.

\bibitem{ford1957cyclic}
L.~Ford~Jr., ``A cyclic arrangement of m-tuples,'' {\em Report No. P-1071, RAND
  Corp}, 1957.

\bibitem{ruskey1992generating}
F.~Ruskey, C.~Savage, and T.~M.~Y. Wang, ``Generating necklaces,'' {\em Journal
  of Algorithms}, vol.~13, no.~3, pp.~414--430, 1992.

\bibitem{fredricksen1986algorithm}
H.~Fredricksen and I.~J. Kessler, ``An algorithm for generating necklaces of
  beads in two colors,'' {\em Discrete mathematics}, vol.~61, no.~2-3,
  pp.~181--188, 1986.

\bibitem{dragon2016grandmama}
P.~B. Dragon, O.~I. Hernandez, and A.~Williams, ``The grandmama de {Bruijn}
  sequence for binary strings,'' in {\em LATIN 2016: Theoretical Informatics},
  pp.~347--361, Springer, 2016.

\bibitem{dragon2018constructing}
P.~B. Dragon, O.~I. Hernandez, J.~Sawada, A.~Williams, and D.~Wong,
  ``Constructing de {Bruijn} sequences with co-lexicographic order: The k-ary
  grandmama sequence,'' {\em European Journal of Combinatorics}, vol.~72,
  pp.~1--11, 2018.

\bibitem{gabric2017bruijn}
D.~Gabric and J.~Sawada, ``A de {Bruijn} sequence construction by concatenating
  cycles of the complemented cycling register,'' in {\em International
  Conference on Combinatorics on Words}, pp.~49--58, Springer, 2017.

\bibitem{gabric2018constructing}
D.~Gabric and J.~Sawada, ``Constructing de {Bruijn} sequences by concatenating
  smaller universal cycles,'' {\em Theoretical Computer Science}, vol.~743,
  pp.~12--22, 2018.

\bibitem{mitchell1996method}
C.~J. Mitchell, T.~Etzion, and K.~G. Paterson, ``A method for constructing
  decodable de {Bruijn} sequences,'' {\em IEEE Transactions on Information
  Theory}, vol.~42, no.~5, pp.~1472--1478, 1996.

\bibitem{lempel1970homomorphism}
A.~Lempel, ``On a homomorphism of the de {Bruijn} graph and its applications to
  the design of feedback shift registers,'' {\em IEEE Transactions on
  Computers}, vol.~100, no.~12, pp.~1204--1209, 1970.

\bibitem{annexstein1997generating}
F.~S. Annexstein, ``Generating de {Bruijn} sequences: An efficient
  implementation,'' {\em IEEE Transactions on Computers}, vol.~46, no.~2,
  pp.~198--200, 1997.

\bibitem{chang1999efficient}
T.~Chang, B.~Park, Y.~H. Kim, and I.~Song, ``An efficient implementation of the
  {D}-homomorphism for generation of de {Bruijn} sequences,'' {\em IEEE
  Transactions on Information Theory}, vol.~45, no.~4, pp.~1280--1283, 1999.

\bibitem{ronse1984feedback}
C.~Ronse, ``Feedback shift registers,'' {\em Lecture Notes in Computer
  Science}, vol.~169, 1984.

\bibitem{tuliani2001bruijn}
J.~Tuliani, ``{De} {Bruijn} sequences with efficient decoding algorithms,''
  {\em Discrete Mathematics}, vol.~226, no.~1-3, pp.~313--336, 2001.

\bibitem{alhakim2011recursive}
A.~Alhakim and M.~Akinwande, ``A recursive construction of nonbinary de
  {Bruijn} sequences,'' {\em Designs, Codes and Cryptography}, vol.~60, no.~2,
  pp.~155--169, 2011.

\bibitem{games1983generalized}
R.~Games, ``A generalized recursive construction for de {Bruijn} sequences,''
  {\em IEEE transactions on information theory}, vol.~29, no.~6, pp.~843--850,
  1983.

\bibitem{alhakim2017stretching}
A.~Alhakim and M.~Nouiehed, ``Stretching de {Bruijn} sequences,'' {\em Designs,
  Codes and Cryptography}, vol.~85, no.~2, pp.~381--394, 2017.

\bibitem{nellore2021invertible}
A.~Nellore, A.~Nguyen, and R.~F. Thompson, ``An invertible transform for
  efficient string matching in labeled digraphs,'' in {\em 32nd Annual
  Symposium on Combinatorial Pattern Matching (CPM 2021)} (P.~Gawrychowski and
  T.~Starikovskaya, eds.), vol.~191 of {\em Leibniz International Proceedings
  in Informatics (LIPIcs)}, (Dagstuhl, Germany), pp.~20:1--20:14, Schloss
  Dagstuhl -- Leibniz-Zentrum f{\"u}r Informatik, 2021.

\bibitem{burrows1994block}
M.~Burrows and D.~J. Wheeler, ``A block-sorting lossless data compression
  algorithm,'' tech. rep., Systems Research Center, 1994.

\bibitem{mykkeltveit1979cycle}
J.~Mykkeltveit, M.-K. Siu, and P.~Tong, ``On the cycle structure of some
  nonlinear shift register sequences,'' {\em Information and control}, vol.~43,
  no.~2, pp.~202--215, 1979.

\bibitem{siu1980generation}
M.-K. Siu and P.~Tong, ``Generation of some de {Bruijn} sequences,'' {\em
  Discrete Mathematics}, vol.~31, no.~1, pp.~97--100, 1980.

\bibitem{hurlbert1996equivalence}
G.~Hurlbert and G.~Isaak, ``Equivalence class universal cycles for
  permutations,'' {\em Discrete Mathematics}, vol.~149, no.~1-3, pp.~123--129,
  1996.

\bibitem{johnson2009universal}
J.~R. Johnson, ``Universal cycles for permutations,'' {\em Discrete
  Mathematics}, vol.~309, no.~17, pp.~5264--5270, 2009.

\bibitem{sawada2014lexicographically}
J.~Sawada, A.~Williams, and D.~Wong, ``The lexicographically smallest universal
  cycle for binary strings with minimum specified weight,'' {\em Journal of
  Discrete Algorithms}, vol.~28, pp.~31--40, 2014.

\bibitem{sawada2016generalizing}
J.~Sawada, A.~Williams, and D.~Wong, ``Generalizing the classic greedy and
  necklace constructions of de {Bruijn} sequences and universal cycles,'' {\em
  the electronic journal of combinatorics}, pp.~P1--24, 2016.

\bibitem{wong2017new}
D.~Wong, ``A new universal cycle for permutations,'' {\em Graphs and
  Combinatorics}, vol.~33, no.~6, pp.~1393--1399, 2017.

\bibitem{sawada2020efficient}
J.~Sawada and D.~Wong, ``Efficient universal cycle constructions for weak
  orders,'' {\em Discrete Mathematics}, vol.~343, no.~10, p.~112022, 2020.

\bibitem{chung1992universal}
F.~Chung, P.~Diaconis, and R.~Graham, ``Universal cycles for combinatorial
  structures,'' {\em Discrete Mathematics}, vol.~110, no.~1-3, pp.~43--59,
  1992.

\bibitem{paterson1995storage}
K.~G. Paterson and M.~J. Robshaw, ``Storage efficient decoding for a class of
  binary de {Bruijn} sequences,'' {\em Discrete mathematics}, vol.~138,
  no.~1-3, pp.~327--341, 1995.

\bibitem{kociumaka2016efficient}
T.~Kociumaka, J.~Radoszewski, and W.~Rytter, ``Efficient ranking of {Lyndon}
  words and decoding lexicographically minimal de {Bruijn} sequence,'' {\em
  SIAM Journal on Discrete Mathematics}, vol.~30, no.~4, pp.~2027--2046, 2016.

\bibitem{sawada2017practical}
J.~Sawada and A.~Williams, ``Practical algorithms to rank necklaces, {Lyndon}
  words, and de {Bruijn} sequences,'' {\em Journal of Discrete Algorithms},
  vol.~43, pp.~95--110, 2017.

\bibitem{mitchell2021constructing}
C.~J. Mitchell and P.~B. Wild, ``Constructing orientable sequences,'' {\em
  arXiv preprint arXiv:2108.03069}, 2021.

\bibitem{burns1992coding}
J.~Burns and C.~J. Mitchell, {\em Coding Schemes for Two-Dimensional Position
  Sensing}.
\newblock Hewlett-Packard Laboratories, Technical Publications Department,
  1992.

\bibitem{dai1993orientable}
Z.~Dai, K.~Martin, M.~Robshaw, and P.~Wild, ``Orientable sequences,'' in {\em
  Institute of Mathematics and Its Applications Conference Series}, vol.~45,
  pp.~97--97, Oxford University Press, 1993.

\bibitem{kak2000yamatarajabhanasalagam}
S.~Kak, ``Yamatarajabhanasalagam: An interesting combinatoric sutra,'' {\em
  Indian Journal of History of Science}, vol.~35, no.~2, pp.~123--128, 2000.

\bibitem{brown1869sanskrit}
C.~P. Brown, {\em Sanskrit Prosody and Numerical Symbols Explained}.
\newblock Tr{\"u}bner \& Company, 1869.

\bibitem{vasu1897ashtadhyayi}
S.~C. Vasu {\em et~al.}, {\em The Ashtadhyayi of Panini}, vol.~6.
\newblock Satyajnan Chaterji, 1897.

\bibitem{laisant1894intermediaire}
C.-A. Laisant and E.~M.~H. Lemoine, {\em L'Interm{\'e}diaire des
  Mathematiciens}, vol.~1.
\newblock Gauthier-Villars et Fils, 1894.

\bibitem{camillebio}
``{Camille} {Flye} {Sainte-Marie}.''
  \url{http://henripoincare.fr/s/correspondance/item/14609}.
\newblock Accessed: 2021-07-23.

\bibitem{debruijnbio}
``{Nicolaas} de {Bruijn} - biography.''
  \url{https://mathshistory.st-andrews.ac.uk/Biographies/De_Bruijn/}.
\newblock Accessed: 2021-07-23.

\bibitem{de1985memoriam}
N.~de~Bruijn, ``In memoriam {T.} van {Aardenne}-{Ehrenfest}, 1905-1984,'' {\em
  Nieuw Archief voor Wiskunde}, vol.~4, no.~2, pp.~235--236, 1985.

\bibitem{hersh2016mathematical}
P.~Hersh, T.~Lam, P.~Pylyavskyy, and V.~Reiner, {\em The Mathematical Legacy of
  {Richard} {P.} {Stanley}}, vol.~100.
\newblock American Mathematical Soc., 2016.

\bibitem{mantel1897resten}
W.~Mantel, ``Resten van wederkerige reeksen,'' {\em Niew Archief voor
  Wiskunde}, vol.~1, pp.~172--184, 1897.

\end{thebibliography}





% \end{thebibliography}

\end{document}